\documentclass[smallextended]{amsart}
\usepackage{amssymb}
\usepackage{mathrsfs}
\usepackage{}
\usepackage{bbm}
\usepackage{amsfonts}
\usepackage{amsfonts,amssymb,amsmath,amsthm}
\usepackage{url}
\usepackage{enumerate}
\usepackage{color}
\usepackage[pdftex,bookmarksnumbered]{hyperref}
\usepackage[dvips]{epsfig}
\usepackage{latexsym}

\DeclareMathOperator{\diam}{diam}
\DeclareMathOperator{\vol}{vol}

\newtheorem{theorem}{Theorem}[section]
\newtheorem{lemma}[theorem]{Lemma}
\newtheorem{prop}[theorem]{Proposition}

\theoremstyle{definition}

\newtheorem{remark}{Remark}

\author{Wei Zhao}
\address{Department of Mathematics\\
East China University of Science and Technology\\
Shanghai 200237, PR China}
\email{szhao\underline{ }wei@yahoo.com}

\begin{document}

\title[Integral curvature bounds and diameter estimates on Finsler manifolds]{Integral curvature bounds and diameter estimates on Finsler manifolds}

\begin{abstract}
In this paper, we study the integral curvatures of Finsler manifolds and prove several Myers type theorems.
\end{abstract}
\maketitle

\section{Introduction}

The Myers theorem is one of the earliest and most fundamental theorems relating geometry
to topology and moreover,
 it also has a close connection with general relativity (cf. \cite{BE,HE}). In Finsler geometry, the Myers theorem  states that if  a forward complete Finsler $n$-manifold with   $\mathbf{Ric}\geq (n-1)K>0$, then
  $M$ is compact with $\diam(M)\leq \pi/\sqrt{K}$ (cf. \cite{BCS,Sh1}). There have been several subsequent generalizations
of this result, among which Ohta \cite{O} proved $\diam(M)\leq \pi\sqrt{{(N-1)}/{K}}$ if $\mathbf{Ric}_N \geq  K > 0$, while Yin \cite{Y} obtained similar compactness results under $\mathbf{Ric}_\infty\geq (n-1)K>0$ and additional assumptions on non-Romanian quantities.
These results   rely on uniformly positive (lower) curvature bounds.

Note that a Riemannian manifold  is a special Finsler manifold. And the works of Ambrose, Calabi, Avez,
Markvorsen, Galloway, Cheeger-Gromov-Taylor, Itokawa, Wu, Rosenberg-Yang, Sprouse and Petersen-Sprouse imply that  a Riemannian manifold could be compact if the Ricci curvature is negative in some small places (see \cite{Am,Ca,Av,M,Ga,CGT,I,Wuu,RY,S,PC}).
  Thus, it is natural to ask whether this is still true in Finsler geometry.

As far as we know, the first attempt is made by Wu \cite{Wu2}, in which one assumes positivity for the integral of the Ricci
curvature along all geodesics.
Inspired by  \cite{PC,PW,PW2,S}, we introduce a weaker integral bound for the Ricci curvature and also give an affirmative answer in this paper.

More precisely, let $(M,F)$ be a forward complete Finsler manifold endowed with either the Busemann-Hausdorff measure or the Holmes-Thompson measure $d\mathfrak{m}$. Let $\Lambda_F$ denote the uniformity constant (cf.\,\cite{E}).
Given $q\geq 1$, $K>0$ and $R>0$, set
\[
\overline{\mathcal {K}}_{d\mathfrak{m}}(q,K,R):=\sup_{x\in M}\left(\frac1{\mathfrak{m}(B^+_x(R))}\int_{B^+_x(R)}((n-1)K-\underline{\mathbf{Ric}})^q_+d\mathfrak{m}\right),
\]
where $(\cdot)_+:=\max\{\cdot,0\}$ and $\underline{\mathbf{Ric}}(x):=\min_{y\in T_xM\backslash\{0\}}\mathbf{Ric}(y)/F^2(y)$.
Obviously,  $\overline{\mathcal {K}}_{d\mathfrak{m}}(q,K,R)$ is a kind of $L^q$-Ricci curvature norm. In particular,
$\overline{\mathcal {K}}_{d\mathfrak{m}}(q,K,R)=0$ if and only if $\mathbf{Ric}\geq (n-1)K$.
And $\overline{\mathcal {K}}_{d\mathfrak{m}}(q,K,R)$ is exactly the integral curvature introduced in \cite{PW,PW2} if $F$ is Riemannian.
Now we obtain the following Myers type theorem by assuming that $\overline{\mathcal {K}}_{d\mathfrak{m}}(q,K,R)$ is small (instead of the positivity of the Ricci curvature).

\begin{theorem}\label{mysfirst}
Given any $n>1$, $q\geq1$, $k\in \mathbb{R}$, $K>0$, $R>0$ and $\delta\geq 1$, for each $\rho>0$, there exists an $\varepsilon=\varepsilon(n,q,k,K,\delta,R,\rho)>0$ such that
every    forward  complete Finsler $n$-manifold $(M,F,d\mathfrak{m})$ with
\[
\mathbf{Ric}\geq -(n-1)k^2,\ \Lambda_F\leq \delta^2,\ \overline{\mathcal {K}}_{d\mathfrak{m}}(q,K,R)<\varepsilon
\]
must satisfy
\[
\diam(M)\leq \frac{\pi}{\sqrt{K}}+\rho.
\]
In particular, the universal covering $\widetilde{M}$ is compact and hence, $\pi_1(M)$ is finite.

Moreover, if additionally suppose that $F$ is Berwaldian and $q>n/2$, then we still have
\[
\diam(M)\leq \delta^2\left( \frac{\pi}{\sqrt{K}}+\rho  \right)
\]
without assuming $\mathbf{Ric}\geq -(n-1)k^2$.
\end{theorem}

\section{Preliminaries}\label{section2}
In this section, we recall some definitions and properties about Finsler manifolds. See \cite{BCS,Sh1} for more details.

\subsection{Finsler manifolds}
A Finsler $n$-manifold $(M,F)$ is an $n$-dimensional differential manifold $M$ equipped with a Finsler metric $F$ which is a nonnegative function on $TM$ satisfying the following two conditions:

(1) $F$ is positively homogeneous, i.e., $F(\lambda y)=\lambda F(y)$, for any $\lambda>0$ and $y\in TM$;

(2) $F$ is smooth on $TM\backslash\{0\}$ and the Hessian $\frac{1}{2}[F^2]_{y^iy^j}(x,y)$ is positive definite, where $F(x,y):=F(y^i\frac{\partial}{\partial x^i}|_x)$.

Let $\pi:PM\rightarrow M$ and $\pi^*TM$ be the projective sphere bundle and the pullback bundle, respectively. Then
a Finsler metric $F$ induces a natural Riemannian metric $g=g_{ij}(x,[y])\,d\mathfrak{x}^i\otimes d\mathfrak{x}^j$, which is  the so-called {\it fundamental tensor}, on  $\pi^*TM$, where
\[
g_{ij}(x,[y]):=\frac12\frac{\partial^2
F^2(x,y)}{\partial y^i\partial
y^j}, \ d\mathfrak{x}^i=\pi^*d x^i.
\]

Egloff \cite{E} introduced the   the {\it uniformity constant} $\Lambda_F$ to describe the inner products induced by $g_y$. More precisely,   set
\[
 \Lambda_F:=\underset{X,Y,Z\in SM}{\sup}\frac{g_X(Y,Y)}{g_Z(Y,Y)},
\]
where $S_xM:=\{y\in T_xM:F(x,y)=1\}$ and $SM:=\cup_{x\in M}S_xM$. Clearly, ${\Lambda_F} \geq  1$ with equality if and only if $F $ is Riemannian.

The {\it average Riemannian metric} $\hat{g}$ induced by $F$  is defined by
\[
\hat{g}(X,Y):=\frac{1}{\nu(S_xM)}\int_{S_xM}g_y(X,Y)d\nu_x(y),\ \forall\,X,Y\in T_xM,
\]
where $\nu(S_xM)=\int_{S_xM}d\nu_x(y)$, and $d\nu_x$ is the Riemannian volume form of $S_xM$ induced by $F$.
It is noticeable that
\[
\Lambda^{-1}_F\cdot F^2(X)\leq \hat{g}(X,X)\leq \Lambda_F \cdot F^2(X),\tag{2.1}\label{2.000}
\]
with equality if and only if $F$ is Riemannian.

\subsection{Curvatures}
Let $(x,y)=(x^i,y^i)$ be
local homogenous coordinates for $PM$. Define
\begin{align*}
&\ell^i:=\frac{y^i}{F},\ g_{ij}(x,y):=\frac12\frac{\partial^2
F^2(x,y)}{\partial y^i\partial
y^j},&A_{ijk}(x,y):=\frac{F}{4}\frac{\partial^3 F^2(x,y)}{\partial
y^i\partial y^j\partial y^k},\\
&\gamma^i_{jk}:=\frac12 g^{il}\left(\frac{\partial g_{jl}}{\partial
x^k}+\frac{\partial g_{kl}}{\partial x^j}-\frac{\partial
g_{jk}}{\partial x^l}\right),
&N^i_j:=\left(\gamma^i_{jk}\ell^j-A^i_{jk}\gamma^k_{rs}\ell^r
\ell^s\right)\cdot F.
\end{align*}
The {\it Chern connection} $\nabla$ is defined on the pulled-back bundle
$\pi^*TM$ and its forms are characterized by the following structure
equations:

(1) Torsion freeness: $dx^j\wedge\omega^i_j=0$;

(2) Almost $g$-compatibility: $d
g_{ij}-g_{kj}\omega^k_i-g_{ik}\omega^k_j=2\frac{A_{ijk}}{F}(dy^k+N^k_l
dx^l)$.

\noindent From above, it's easy to obtain $\omega^i_j=\Gamma^i_{jk}dx^k$, and
$\Gamma^i_{jk}=\Gamma^i_{kj}$.  In particular, $F$ is called a {\it Berwald metric} if ${\partial\Gamma^i_{kj}}/{\partial y^s}=0$.

The {\it curvature form} of the Chern connection is defined as
\[
\Omega^i_j:=d\omega^i_j-\omega^k_j\wedge\omega^i_k=:\frac{1}{2}R^i_{j\,kl}d
x^k\wedge d x^l+P^i_{j\,kl}d x^k\wedge\frac{d y^l+N^l_s dx^s}{F}.
\]
Given a non-zero vector $V\in T_xM$, the {\it flag curvature} $\mathbf{K}(y,V)$ on
$(x,y)\in TM\backslash 0$ is defined as
\[
\mathbf{K}(y,V):=\frac{V^i y^jR_{jikl}(y)y^l
V^k}{g_y(y,y)g_y(V,V)-[g_y(y,V)]^2},
\]
where $R_{jikl}:=g_{is}R^s_{j\,kl}$.
The {\it Ricci curvature} of $y\in SM$
is defined by
\[
\mathbf{Ric}(y):=\underset{i}{\sum}\,K(y,e_i),
\]
where $e_1,\ldots, e_n$ is a $g_y$-orthonormal base on $(x,y)\in
TM\backslash0$. We also use the notation
\[
\underline{\mathbf{Ric}}(x):=\min_{y\in S_xM}\mathbf{Ric}(y).
\]

\subsection{Geodesics}
Let $\gamma:[0,1]\rightarrow M$ be a Lipschitz continuity path. The length of $\gamma$ is defined by
\[
L_F(\gamma):=\int^1_0 F(\dot{\gamma}(t))dt.
\]
Define the distance function $d:M\times M\rightarrow [0,+\infty)$ by
$d(p,q):=\inf L_F(\gamma)$,
where the infimum is taken over all
Lipshitz continuous paths $\gamma:[a,b]\rightarrow M$ with
$\gamma(a)=p$ and $\gamma(b)=q$. Generally speaking, $d$ satisfies all axioms for a metric except symmetry, i.e., $d(p,q)\neq d(q,p)$, unless $F$ is reversible.

Given $R>0$, the {\it forward} and {\it backward metric balls} $B^+_p(R)$ and $B^-_p(R)$ are defined by
\[
B^+_p(R):=\{x\in M:\, d(p,x)<R\},\ B^-_p(R):=\{x\in M:\, d(x,p)<R\}.
\]
If $F$ is reversible, forward metric balls coincide with backward ones.

A smooth curve $\gamma(t)$ is called a (constant speed) {\it geodesic} if it satisfies
\[
\ddot{\gamma}^i+\Gamma^i_{jk}(\dot{\gamma})\dot{\gamma}^j\dot{\gamma}^k=0.
\]
In the paper, we always to use $\gamma_y(t)$ to denote the geodesic with $\dot{\gamma}_y(0)=y$.
Note that the reverse of a geodesic is usually not a geodesic unless $F$ is reversible.

Given $y\in S_pM$, the {\it cut value} $i_y$ of $y$ is defined by
\[
i_y:=\sup\{r: \,\gamma_y(t), \,0\leq t\leq r \text{ is globally minimizing}  \}.
\]
The {\it injectivity radius} at $p$ is defined as $\mathfrak{i}_p:=\inf_{y\in S_pM} i_y$, whereas the {\it cut locus} of $p$ is
\[
\text{Cut}_p:=\left\{\exp_p(i_y\cdot y):\,y\in S_pM \text{ with }i_y<\infty \right\}.
\]
It should be remarked that $\text{Cut}_p$ is closed and null Lebesgue measure.

\subsection{Measures}

There is only one reasonable notion of the measure for
Riemannian manifolds. However, the measures on a Finsler manifold can be defined in various ways, since the determinant of the fundamental tensor  depends on the direction of $y$. There are  two measures used frequently in Finsler geometry, which are the so-called
 {\it Busemann-Hausdorff measure} $d\mathfrak{m}_{BH}$ and {\it Holmes-Thompson measure} $d\mathfrak{m}_{HT}$. They are defined by
\begin{align*}
&d\mathfrak{m}_{BH}:=\frac{\vol(\mathbb{B}^{n})}{\vol(B_xM)}dx^1\wedge\cdots\wedge dx^n,\\
 &d\mathfrak{m}_{HT}:=\left(\frac1{\vol(\mathbb{B}^{n})}\int_{B_xM}\det g_{ij}(x,y)dy^1\wedge\cdots\wedge dy^n \right) dx^1\wedge\cdots\wedge dx^n,
\end{align*}
where $B_xM:=\{y\in T_xM: F(x,y)<1\}$. Each of them becomes the canonical Riemannian measure if $F$ is Riemannian. However, their properties are different.  Even in the reversible case,
$d\mathfrak{m}_{BH}\leq d\mathfrak{m}_{HT}$ with equality if and only if $F$ is Riemannian (cf. \cite{AlB}).

Let $d\mathfrak{m}$ be a measure on $M$. In a local coordinate system $(x^i)$,
express $d\mathfrak{m}=\sigma(x)dx^1\wedge\cdots\wedge dx^n$.
Define the {\it distortion} of $(M,F,d\mathfrak{m})$ as
\begin{equation*}
\tau(y):=\log \frac{\sqrt{\det g_{ij}(x,y)}}{\sigma(x)}, \text{ for $y\in T_xM\backslash\{0\}$}.
\end{equation*}
And the {\it S-curvature} $\mathbf{S}$ is defined by
\begin{equation*}
\mathbf{S}(y):=\frac{d}{dt}[\tau(\dot{\gamma}_y(t))]|_{t=0}.
\end{equation*}
It is easy to see both the distortion and the S-curvature vanish in the Riemannian case.
According to \cite{Sh1,ZY3}, the S-curvatures of $d\mathfrak{m}_{BH}$ and $d\mathfrak{m}_{HT}$ always vanish when $F$ is Berwaldian.
The following result is useful in this paper (cf. \cite[Lemma 2.1]{YZ}).
\begin{prop}\label{uniformtorsion}
Let $(M,F)$ be a Finsler $n$-manifold with $\Lambda_F\leq \delta^2$. Then we have
$\delta^{-2n}\leq e^\tau\leq \delta^{2n}$, where $\tau$ is the distortion of either the Busemann-Hausdorff measure or the Holmes-Thompson measure.
\end{prop}

In the following, $v(n,k,r)$ is used to denote the volume of $r$-ball in the space form $\mathbb{M}^n_k$, i.e.,
\[
v(n,k,r):=\vol({\mathbb{S}^{n-1}})\int^r_0\mathfrak{s}^{n-1}_k(t)dt.
\]

\section{Proof of Theorem \ref{mysfirst}}

In the sequel, we  use $d\mathfrak{m}$ to denote either the Busemann-Hausdorff measure or the Holmes-Thompson measure.

In order to prove Theorem \ref{mysfirst}, we need to generalize the so-called segment inequality \cite[Theorem 2.1]{CC} to the Finsler setting.

\begin{theorem}\label{CheegerColding}
Let $(M,F,d\mathfrak{m} )$ be a forward complete Finsler $n$-manifold with
\[
\Lambda_F\leq \delta^2,\ \mathbf{Ric}\geq (n-1)k.
\]
Let $A_i$, $i=1,2$ be two bounded open subsets and let $W$ be an open subset such that for each two points $x_i\in A_i$, $i=1,2$, the normal minimal geodesic $\gamma_{x_1x_2}$ from $x_1$ to $x_2$ is contained in $W$. Thus, for any non-negative integrable function $f$ on $W$, we have
\begin{align*}
&\int_{A_1\times A_2}\left(\int^{d(x_1,x_2)}_0 f(\gamma_{x_1x_2}(s))ds\right) d\mathfrak{m}_{\times}\\
\leq& C(n,\delta,k,D)\left[ \mathfrak{m}(A_1)\diam(A_2)+\mathfrak{m}(A_2)\diam(A_1)\right]\int_W f d\mathfrak{m},
\end{align*}
where $d\mathfrak{m}_{\times}$ is the product measure induced by $d\mathfrak{m}$, $D:=\sup_{x_1\in A_1,\, x_2\in A_2}d(x_1,x_2)$ and
\[
C(n,k,\delta,D)=\delta^{4n}\sup_{0<\frac12r\leq s\leq r\leq D}\left(\frac{\mathfrak{s}_k(r)}{\mathfrak{s}_k(s)}\right)^{n-1}.
\]
\end{theorem}
\begin{proof}
\textbf{Step 1.}  Set
$B:=\{(x_1,x_2)\in A_1\times A_2: \ \text{there exists a unique minimal geodesic}$ $\text{from $x_1$ to $x_2$}\}$.
Since  $\mathfrak{m}_{\times}(B)=\mathfrak{m}_{\times}(A_1\times A_2)$, in the following we  view $A_1\times A_2$ as $B$. Let
\begin{align*}
E(x_1,x_2)&:=\int_0^{d(x_1,x_2)} f(\gamma_{x_1x_2}(s))ds=\int_{\frac12 d(x_1,x_2)}^{d(x_1,x_2)}+\int_0^{\frac12 d(x_1,x_2)}  f(\gamma_{x_1x_2}(s))ds\\
&=:E_1(x_1,x_2)+E_2(x_1,x_2).\tag{3.1}\label{3.1newnewnew}
\end{align*}
Given $y\in S_{x_1}M$, set
\[
I(x_1,y):=\{t:\, \gamma_{y}(t)\in A_2,\ \gamma_y|_{[0,t]}\text{ is minimal} \}.
\]
Clearly, the length of $I(x_1,y)$ is not larger than $\diam(A_2)$.
Set
\[
T(y):=\sup\{t:\, t\in I(x_1,y) \}.
\]
Let $(r,y)$ denote the polar coordinate system at $x_1$.
Since $\mathbf{Ric}\geq (n-1)k$, the volume comparison theorem (cf.\,\cite[Theorem 3.4]{ZY2}) yields
\[
\frac{e^{\tau(s,y)}\hat{\sigma}_{x_1}(s,y)}{\mathfrak{s}^{n-1}_k(s)}\geq \frac{e^{\tau(r,y)}\hat{\sigma}_{x_1}(r,y)}{\mathfrak{s}^{n-1}_k(r)},\ 0<s\leq r<i_y,
\]
which together with Proposition \ref{uniformtorsion} yields
\begin{align*}
&\int_{x_2\in A_2}E_1(x_1,x_2)d\mathfrak{m}(x_2)=\int_{S_{x_1}M} d\nu_{x_1}(y)\int_{I(x_1,y)}E_1(x_1,\exp_{x_1}(ry))\cdot \hat{\sigma}_{x_1}(r,y)\, dr\\
=&\int_{S_{x_1}M} d\nu_{x_1}(y)\int_{I(x_1,y)}\hat{\sigma}_{x_1}(r,y)dr\int^{r}_{\frac12 r}f(\gamma_{x_1 \exp_{x_1}(ry)}(s)) ds\\
\leq &C(n,k,\delta,D)\int_{S_{x_1}M} d\nu_{x_1}(y)\int_{I(x_1,y)}dr\int^{r}_{\frac12 r}f(\gamma_{x_1 \exp_{x_1}(ry)}(s))\, \hat{\sigma}_{x_1}(s,y) ds\\
\leq &C(n,k,\delta,D)\diam(A_2)\int_{S_{x_1}M} d\nu_{x_1}(y)\int^{T(y)}_{0}f(\gamma_{x_1 \exp_{x_1}(ry)}(s))\, \hat{\sigma}_{x_1}(s,y) ds\\
\leq &C(n,k,\delta,D)\diam(A_2)\int_{W}f d\mathfrak{m}.
\end{align*}
Hence, we have
\[
\int_{A_1\times A_2}E_1(x_1,x_2)d\mathfrak{m}_{\times}\leq C(n,k,\delta,D)\cdot\diam(A_2)\cdot \mathfrak{m}(A_1)\int_{W}f d\mathfrak{m}.\tag{3.2}\label{3.2newnewnew}
\]

\noindent\textbf{Step 2.} In this step, we estimate
\[
E_2(x_1,x_2)=\int_0^{\frac12 d(x_1,x_2)}  f(\gamma_{x_1x_2}(s))ds.
\]
Let $\overleftarrow{F}(x,y):=F(x,-y)$ be the reverse metric (cf. \cite{BCS,O}). In the following, we use $\overleftarrow{*}$ to denote the geometric quantity $*$ in $(M,\overleftarrow{F})$.
 Let $c_{x_2x_1}(s)$ be the reverse of $\gamma_{x_1x_2}$. Thus, $c_{x_2x_1}(s)$ is a normal minimal geodesic from $x_2$ to $x_1$ in $(M,\overleftarrow{F})$ with the length $\overleftarrow{d}(x_2,x_1)=d(x_1,x_2)$. Hence,
\begin{align*}
E_2(x_1,x_2)=\int_0^{\frac12 d(x_1,x_2)}  f(\gamma_{x_1x_2}(s))ds=\int^{\overleftarrow{d}(x_2,x_1)}_{\frac12\overleftarrow{d}(x_2,x_1)} f(c_{x_2x_1}(s))ds=:\overleftarrow{E}_1(x_2,x_1).
\end{align*}
Note that $\overleftarrow{\mathbf{Ric}}\geq (n-1)k$. The same argument as above yields
\begin{align*}
\int_{A_2\times A_1}\overleftarrow{E}_1(x_2,x_1)d\overleftarrow{\mathfrak{m}}_{\times}\leq C(n,k,\delta,D)\cdot\overleftarrow{\diam}(A_1)\cdot \overleftarrow{\mathfrak{m}}(A_2)\int_{W}f d\overleftarrow{\mathfrak{m}}.
\end{align*}
On the other hand, it is easy to check that $d\overleftarrow{\mathfrak{m}}=d \mathfrak{m}$, $\overleftarrow{\mathfrak{m}}(A_2)={\mathfrak{m}}(A_2)$ and $\overleftarrow{\diam}(A_1)=\diam(A_1)$, which implies
\[
\int_{A_1\times A_2} E_2(x_1,x_2) d\mathfrak{m}_{\times}\leq C(n,k,\delta,D)\cdot{\diam}(A_1)\cdot {\mathfrak{m}}(A_2)\int_{W}f d{\mathfrak{m}}.\tag{3.3}\label{3.3newnewnew}
\]
Now we conclude the proof by (\ref{3.1newnewnew})-(\ref{3.3newnewnew}).
\end{proof}
\begin{remark}
By the comparison theorem in \cite{O}, one can see that the theorem above remains valid under a simpler assumption $\mathbf{Ric}_N\geq(N-1)k$ instead of $\Lambda_F\leq \delta^2,\ \mathbf{Ric}\geq (n-1)k$, in which case $C(n,k,\delta,D)$ is replaced by 
\[
C(N,k,D):=\sup_{0<\frac12r\leq s\leq r\leq D}\left(\frac{\mathfrak{s}_k(r)}{\mathfrak{s}_k(s)}\right)^{N-1}.
\]
\end{remark}

 Let $(M,F,d\mathfrak{m})$ be an $n$-dimensional forward complete Finsler manifold.
Given $q\geq 1$, $K>0$ and $R>0$, set
\[
\overline{\mathcal {K}}_{d\mathfrak{m}}(q,K,R):=\sup_{x\in M}\left(\frac1{\mathfrak{m}(B^+_x(R))}\int_{B^+_x(R)}((n-1)K-\underline{\mathbf{Ric}})^q_+d\mathfrak{m}\right),
\]
where $(\cdot)_+:=\max\{\cdot,0\}$. Clearly, $\overline{\mathcal {K}}_{d\mathfrak{m}}(q,K,R)$ is a kind of $L^q$-Ricci curvature norm.

In \cite{S}, Sprouse proved some compactness theorems by $L^1$-Ricci curvature bounds in the Riemannnian case. Inspired by his work, we show the following result by the $L^q$-norm ($\forall q\geq 1$).
\begin{lemma}\label{Rlesslemm}
Given any $n>1$, $q\geq 1$, $k\in \mathbb{R}$, $K>0$ and $\delta\geq 1$,
for each $R>\pi/{\sqrt{K}}$ and each $\rho\in\left(0,\frac{1}{1+\delta}\left( R-\frac{\pi}{\sqrt{K}}\right)\right)$,
 there exists $\varepsilon=\varepsilon(n,q,k,K,\delta,R,\rho)>0$ such that
every  forward   complete Finsler $n$-manifold $(M,F,d\mathfrak{m})$ with
\begin{align*}
\mathbf{Ric}\geq -(n-1)k^2,\ \Lambda_F\leq \delta^2,\ \overline{\mathcal {K}}_{d\mathfrak{m}}(q,K,R)<\varepsilon
\end{align*}
must satisfy
\[
\diam(M)\leq \frac{\pi}{\sqrt{K}}+\rho.
\]
\end{lemma}
\begin{proof}
\textbf{Step 1.} Fix a point $p\in M$ and set $W=B^+_p(R)$. Choose any point $p'\in W$ with
\[
\frac{\pi}{\sqrt{K}}+\delta(3+2\delta)r<d(p,p')<R-(2+\delta)r,\tag{3.4}\label{newnew4.1}
\]
where $r=\frac{\rho}{2(1+\delta)}$ is a fixed number.

Set $A_1:=B^+_p(r)$ and $A_2:=B^+_{p'}(r)$. By the triangle inequality, one can easily show that $A_1,A_2,W$ satisfy the conditions in Theorem \ref{CheegerColding}.
Thus, Theorem \ref{CheegerColding} yields
\begin{align*}
 &\int_{A_1\times A_2}\left(\int^{d(x_1,x_2)}_0 ((n-1)K-\underline{\mathbf{Ric}})^q_+(\gamma_{x_1x_2}(s))ds\right) d\mathfrak{m}_{\times}\\
\leq& C(n,k,\delta,R)\cdot (1+\delta)r\left[ \mathfrak{m}(A_1)+\mathfrak{m}(A_2)\right]\int_W ((n-1)K-\underline{\mathbf{Ric}})^q_+ d\mathfrak{m}.\tag{3.5}\label{3.8CC}
 \end{align*}
The volume comparison theorem (cf.\,\cite[Remark 3.5]{ZY2}) together with Proposition \ref{uniformtorsion} yields
\begin{align*}
\frac{\mathfrak{m}(A_1)}{\mathfrak{m}(W)}\geq \delta^{-4n}\frac{v(n,-k^2,r)}{v(n,-k^2,R)},\  \frac{\mathfrak{m}(A_2)}{\mathfrak{m}(B^+_{p'}((1+\delta)R))}\geq \delta^{-4n}\frac{v(n,-k^2,r)}{v(n,-k^2,(1+\delta)R)}.
\end{align*}
Since $\mathfrak{m}(B^+_{p'}((1+\delta)R))\geq \mathfrak{m}(W)$, (\ref{3.8CC}) together with the above inequalities implies
\begin{align*}
&\inf_{(x_1,x_2)\in A_1\times A_2}\int^{d(x_1,x_2)}_0((n-1)K-\underline{\mathbf{Ric}})^q_+(\gamma_{x_1x_2}(s))ds\\
\leq& C(n,k,\delta,R)\cdot (1+\delta)r \cdot \left( \frac{1}{\mathfrak{m}(A_1)} +\frac{1}{\mathfrak{m}(A_2)}  \right)\int_{W} ((n-1)K-\underline{\mathbf{Ric}})^q_+ d\mathfrak{m}\\
\leq &C(n,k,\delta,R,\rho)\cdot \frac{1}{\mathfrak{m}(W)}\int_{W} ((n-1)K-\underline{\mathbf{Ric}})^q_+ d\mathfrak{m}\\
\leq &C(n,k,\delta,R,\rho)\cdot  \overline{\mathcal {K}}_{d\mathfrak{m}}(q,K,R),
\end{align*}
where
\[
C(n,k,\delta,R,\rho):=2\delta^{4n}\cdot C(n,k,\delta,R)\cdot (1+\delta)r\cdot \frac{v(n,-k^2,(1+\delta)R)}{v(n,-k^2,r)}.
\]

Since $\overline{A_1\times A_2}$ is compact, there exist two points $\tilde{x}_i\in \overline{A}_i$, $i=1,2$ such that the normal minimal geodesic $\gamma_{\tilde{x}_1\tilde{x}_2}(t)$ satisfies
\begin{align*}
&\int^{d(\tilde{x}_1,\tilde{x}_2)}_0((n-1)K-\underline{\mathbf{Ric}})^q_+(\gamma_{\tilde{x}_1\tilde{x}_2}(s))ds\\
=&\inf_{(x_1,x_2)\in A_1\times A_2}\int^{d(x_1,x_2)}_0((n-1)K-\underline{\mathbf{Ric}})^q_+(\gamma_{x_1x_2}(s))ds\\
\leq& C(n,k,\delta,R,\rho)\cdot  \overline{\mathcal {K}}_{d\mathfrak{m}}(q,K,R)<\,C(n,k,\delta,R,\rho)\cdot \varepsilon,\tag{3.6}\label{4.0C}
\end{align*}
where $\varepsilon$ will be chosen in the sequel.

\noindent\textbf{Step 2.} Set $L:=\sqrt{K}\,d(\tilde{x}_1,\tilde{x}_2)$. Let $T:=\dot{\gamma}_{\tilde{x}_1\tilde{x}_2}$ and  $\{E_1(t),\ldots,E_{n-1}(t),T\}$ be a $g_T$-orthonormal   parallel frame field along $\gamma_{\tilde{x}_1\tilde{x}_2}$. Set
\[
Y_\alpha(t):=\sin\left( \frac{\pi \sqrt{K}}{L}t \right)\cdot E_\alpha(t), \ \alpha=1,\ldots,n-1.
\]
Let $C_\alpha(t,s)$ be the fixed-endpoint variation of curves corresponding to $Y_\alpha$ (i.e., $Y_\alpha(t)=(\partial_sC_\alpha)(t,0)$) and $\mathcal {L}_\alpha(s)$ be the length of $C_\alpha(\cdot,s)$. Then we have
\begin{align*}
&\sum_\alpha \left.\frac{d^2}{ds^2}\right|_{t=0}\mathcal {L}_\alpha=\sum_\alpha\int_0^{\frac{L}{\sqrt{K}}}g_T(\nabla_TY_\alpha,\nabla_TY_\alpha)+R_T(T,Y_\alpha,T,Y_\alpha)\,dt\\
=&-\frac{(n-1)L\sqrt{K}}{2}\left( 1-\left(\frac{\pi}{L}\right)^2\right)+\Delta,\tag{3.7}\label{4.1CCCC}
\end{align*}
where
\[
\Delta:=\int_0^{\frac{L}{\sqrt{K}}}\sin^2\left( \frac{\pi \sqrt{K}}L{t} \right)\left[ (n-1)K-\mathbf{Ric}(T) \right]\,dt.
\]
Now using the H\"older inequality and (\ref{4.0C}), one gets
\begin{align*}
\Delta
\leq &\int_0^{\frac{L}{\sqrt{K}}}\sin^2\left( \frac{\pi \sqrt{K}}{L}t \right)\left( (n-1)K-\mathbf{Ric}(T) \right)_+\,dt \\
\leq & \left(\int^{\frac{L}{\sqrt{K}}}_0\sin^{\frac{2q}{q-1}}\left( \frac{\pi \sqrt{K}t}{L} \right) dt \right)^{1-\frac1q}\left(\int^{\frac{L}{\sqrt{K}}}_0\left( (n-1)K-\mathbf{Ric}(T) \right)^q_+   dt \right)^{\frac{1}q}\\
= &C(q,n,k,K,\delta,R,\rho)\cdot L^{1-1/q}\cdot \varepsilon^{\frac1q},\tag{3.8}\label{4.2CCCC}
\end{align*}
where $C(q,n,k,K,\delta,R,\rho):=C(n,k,\delta,R,\rho)^{\frac1q}\cdot K^{\frac{1-q}{2q}}$.

\noindent\textbf{Step 3.} Now we claim $L\leq \pi+\frac{\rho}{2}\sqrt{K}$, if $\varepsilon$ is small enough.
 Suppose by contradiction that  $L> \pi+\frac{\rho}{2}\sqrt{K}$ for any $\varepsilon>0$.
 We consider some $\varepsilon>0$ with
\[
\varepsilon\leq \left[ \frac{(n-1)\sqrt{K}\left(1-\left(\frac{\pi}{\pi+\frac{\rho}{2}\sqrt{K}}\right)^2  \right)}{2\cdot C(q,n,k,K,\delta,R,\rho) }   \right]^q\cdot \left( \pi+\frac{\rho}{2}\sqrt{K} \right)=:\epsilon_1.
\]
It follows from (\ref{4.1CCCC}) and (\ref{4.2CCCC}) that
\[
\sum_\alpha \left.\frac{d^2}{ds^2}\right|_{t=0}\mathcal {L}_\alpha=-\frac{(n-1)L\sqrt{K}}{2}\left( 1-\left(\frac{\pi}{L}\right)^2\right)+\Delta<0,
\]
which is a contradiction, since $\gamma$ is a minimal geodesic. Hence, the claim is true and then the triangle inequality implies that
\begin{align*}
d(p,p')\leq d(p,\tilde{x}_1)+d(\tilde{x}_1,\tilde{x}_2)+d(\tilde{x}_2,p')\leq (1+\delta)r+\frac{L}{\sqrt{K}}<\frac{\pi}{\sqrt{K}}+\rho.
\end{align*}
Recall that $p'$ is an arbitrary point satisfying (\ref{newnew4.1}) and hence,
\[
B^+_p\left(R-(2+\delta)r\right)\subset B^+_p\left( \frac{\pi}{\sqrt{K}}+\rho \right).
\]
However, it is easy to check that
\[
R-(2+\delta)r>\frac{\pi}{\sqrt{K}}+\rho\Longrightarrow M=B^+_p\left( \frac{\pi}{\sqrt{K}}+\rho \right).
\]
In particular, $M$ is compact.

\noindent\textbf{Step 4.}  Now  we  estimate $\diam(M)$.  Since $M$ is compact, we can suppose that there exist two points $p,p'\in M$ such that
\[
D:=\diam(M)=d(p,p')>\frac{1}{1+\delta}\left( \frac{\pi}{\sqrt{K}}+\rho  \right).
\]
Otherwise, we are done.
Fix a number $r$ with
\[
0<r<\min\left\{  \frac{1}{(1+\delta)^2}\left( \frac{\pi}{\sqrt{K}}+\rho  \right),\ \frac{\rho}{2(1+\delta)}  \right\}
\]
and set
\[
R_0:=\frac{\pi}{\sqrt{K}}+\rho,\ A_1:=B^+_p(r),\ A_2:=B^+_{p'}(r),\ W:=M=B^+_p(R_0)=B^+_p(R).
\]
Since $D<(1+\delta)R_0$, the same argument as above (see (\ref{4.0C})) yields that there exist two points $\tilde{x}_i\in A_i$, $i=1,2$ with
\begin{align*}
&\int^{d(\tilde{x}_1,\tilde{x}_2)}_0((n-1)K-\underline{\mathbf{Ric}})^q_+(\gamma_{\tilde{x}_1\tilde{x}_2}(s))ds\\
=&\inf_{(x_1,x_2)\in A_1\times A_2}\int^{d(x_1,x_2)}_0((n-1)K-\underline{\mathbf{Ric}})^q_+(\gamma_{x_1x_2}(s))ds\\
\leq& C(n,k,\delta,R_0,\rho)\cdot  \frac{1}{\mathfrak{m}(B^+_p(R))}\int_{B^+_p(R)} ((n-1)K-\underline{\mathbf{Ric}})^q_+ d\mathfrak{m}\\
<& C(n,k,\delta,R_0,\rho)\cdot\varepsilon.
\end{align*}
Set $L:=\sqrt{K}\,d(\tilde{x}_1,\tilde{x}_2)$. Using the same arguments in Step 2-3, one can show that if
\[
\varepsilon\leq\left[ \frac{(n-1)\sqrt{K}\left(1-\left(\frac{\pi}{\pi+\frac{\rho}{2}\sqrt{K}}\right)^2  \right)}{2\cdot C(q,n,k,K,\delta,R_0,\rho) }   \right]^q\cdot \left( \pi+\frac{\rho}{2}\sqrt{K} \right)=:\epsilon_2,
\]
then $L\leq \pi+\frac{\rho}{2}\sqrt{K}$. Now, one gets
\[
D=d(p,p')\leq d(p,\tilde{x}_1)+d(\tilde{x}_1,\tilde{x}_2)+d(\tilde{x}_2,p')<\frac{L}{\sqrt{K}}+(1+\delta)r\leq \frac{\pi}{\sqrt{K}}+\rho.
\]
Now we conclude the proof  by choosing $\varepsilon:=\min\{\epsilon_1,\epsilon_2\}$.
\end{proof}

We now recall the definition and properties of fundamental domain. See \cite{GLP} for more details.
Let $f:\widetilde{M}\rightarrow M$ be a covering with deck
transformation group $\Gamma$. $\Omega\subset \widetilde{M}$ is called
{\it a fundamental domain }of $\widetilde{M}$ if $f(\overline{\Omega})=M$ and
$\gamma(\Omega)\cap \Omega=\emptyset$, for all $\gamma \in
\Gamma-\{1\}$.
If $\Omega$ is a fundamental domain, then
\[
\underset{\gamma\in \Gamma}%
{\cup}\gamma(\overline{\Omega})=\widetilde{M},\ f|_{\gamma(\Omega)}:%
\gamma(\Omega)\rightarrow f(\Omega)\text{ is a homeomorphism},\ \forall\gamma\in \Gamma.
\]

If $(M,F)$ is forward complete, one can get a fundamental domain as follows. For any $p\in M$,
\begin{equation*}
p\mapsto D_p\subset T_{{p}}{M}\mapsto f_*|_{\tilde{p}}^{-1}(D_p)\subset T_{%
\tilde{p}}\widetilde{M}\mapsto \exp_{\tilde{p}}(f_*|_{\tilde{p}%
}^{-1}(D_p))=:\Omega_p,\tag{3.9}\label{3.12**C}
\end{equation*}
where $\tilde{p}$ is an arbitrary point
in $f^{-1}(p)$. Thus,
 $\Omega_p$ is a fundamental domain.

On the other hand, if $d\mathfrak{m}$ is either the Busemann-Hausdorff measure or the Holmes-Thompson measure, then the pull-back measure $f^*d\mathfrak{m}$ is exactly the same kind of measure on $(\widetilde{M},f^*F)$.  By abuse of notation, $d\mathfrak{m}$ also denotes the pull-back measure.

\begin{theorem}\label{firstheorem}Given any $n>1$, $q\geq1$, $k\in \mathbb{R}$, $K>0$, $R>0$ and $\delta\geq 1$, for each $\rho>0$, there exists $\varepsilon=\varepsilon(n,q,k,K,\delta,R,\rho)>0$ such that
every    forward  complete Finsler $n$-manifold $(M,F)$ with
\[
\mathbf{Ric}\geq -(n-1)k^2,\ \Lambda_F\leq \delta^2,\ \overline{\mathcal {K}}_{d\mathfrak{m}}(q,K,R)<\varepsilon
\]
must satisfy
\[
\diam(M)\leq \frac{\pi}{\sqrt{K}}+\rho.
\]
In particular, the universal covering $\widetilde{M}$ is compact and hence, $\pi_1(M)$ is finite.
\end{theorem}
\begin{proof} \textbf{Step 1.} We first estimate the diameter of $M$.
If  $R>{\pi}/{\sqrt{K}}$, then the theorem  follows from Lemma \ref{Rlesslemm} directly. Hence, we only need to consider the case when $R\leq {\pi}/{\sqrt{K}}$.

Fix  $\mathfrak{R}>\pi/\sqrt{K}$ arbitrarily and choose any point $x\in M$. Let $\{B^+_{x_i}(R/(1+\delta))\}_{i\in I}$ be a maximal disjoint family in $B^+_x(\mathfrak{R})$. Thus, $\{B^+_{x_i}(R)\}_{i\in I}$ is a covering of  $B^+_x(\mathfrak{R})$ and hence,
\begin{align*}
&\frac{1}{\mathfrak{m}(B^+_{x}(\mathfrak{R}))}\int_{B^+_{x}(\mathfrak{R})}((n-1)K-\underline{\mathbf{Ric}})^q_+d\mathfrak{m}\\
\leq& \frac{1}{\mathfrak{m}(B^+_{x}(\mathfrak{R}))}\sum_{i\in I} \int_{B^+_{x_i}(R)}((n-1)K-\underline{\mathbf{Ric}})^q_+d\mathfrak{m}\\
=&\sum_{i\in I}\frac{\mathfrak{m}(B^+_{x_i}(R))}{\mathfrak{m}(B^+_{x}(\mathfrak{R}))}\frac{1}{\mathfrak{m}(B^+_{x_i}(R))}\int_{B^+_{x_i}(R)}((n-1)K-\underline{\mathbf{Ric}})^q_+d\mathfrak{m}\\
\leq &\sum_{i\in I}\frac{\mathfrak{m}(B^+_{x_i}(R))}{\mathfrak{m}(B^+_{x}(\mathfrak{R}))}\overline{\mathcal {K}}_{d\mathfrak{m}}(q,K,R)\\
\leq &\sum_{i\in I}\frac{\mathfrak{m}(B^+_{x_i}(R))}{\mathfrak{m}(B^+_{x_i}(R/(1+\delta)))}\frac{\mathfrak{m}(B^+_{x_i}(R/(1+\delta)))}{\mathfrak{m}(B^+_x(\mathfrak{R}))}\overline{\mathcal {K}}_{d\mathfrak{m}}(q,K,R)\\
\leq &C(n,k,\delta,R)\sum_{i\in I}\frac{\mathfrak{m}(B^+_{x_i}(R/(1+\delta)))}{\mathfrak{m}(B^+_x(\mathfrak{R}))}\overline{\mathcal {K}}_{d\mathfrak{m}}(q,K,R)\\
\leq &C(n,k,\delta,R)\cdot\overline{\mathcal {K}}_{d\mathfrak{m}}(q,K,R),
\end{align*}
where
\[
C(n,k,\delta,R):=\delta^{4n}\frac{v(n,-k^2,R)}{v(n,-k^2,R/(1+\delta))}.
\]
That is,
\[
\overline{\mathcal {K}}_{d\mathfrak{m}}(q,K,\mathfrak{R})\leq C(n,k,\delta,R)\cdot\overline{\mathcal {K}}_{d\mathfrak{m}}(q,K,R).\tag{3.10}\label{4.7newnew}
\]
Choose $\varepsilon=\varepsilon(n,q,k,K,\delta,\mathfrak{R},\rho)$ as defined in Lemma \ref{Rlesslemm} and set
\[
\overline{\mathcal {K}}_{d\mathfrak{m}}(q,K,R)<\frac{\varepsilon}{C(n,k,\delta,R)}.
\]
Then the estimate of the diameter follows from (\ref{4.7newnew}) and Lemma \ref{Rlesslemm}.

\noindent \textbf{Step 2.} Now we show that $\widetilde{M}$ is compact. In the sequel, we  use $\widetilde{*}$ to denote the geometric quantity $*$ in $(\widetilde{M},f^*F)$. Suppose that $R>\pi/\sqrt{K}$. It follows from Lemma \ref{Rlesslemm} that $\diam(M)<(1+\delta)R$.  We show that $\widetilde{\overline{{\mathcal {K}}}}_{d\mathfrak{m}}(q,K,R)$ is controlled by $\overline{{\mathcal {K}}}_{d\mathfrak{m}}(q,K,R)$, where  $\widetilde{\overline{{\mathcal {K}}}}_{d\mathfrak{m}}(q,K,R)$ is the integral curvature
of $(\widetilde{M},f^*F)$.

In fact, given any point $\tilde{x}\in \widetilde{M}$, let $N$ denote the minimal number of the fundamental domains $\gamma(\Omega)$ covering $B^+_{\tilde{x}}(R)$, i.e.,
\[
B^+_{\tilde{x}}(R)\subset \cup_{i=1}^N\gamma_i(\Omega)\subset B^+_{\tilde{x}}((2+\delta)R).\tag{3.11}\label{4.3C}
\]
Hence,
\[
\frac{\mathfrak{m}(B^+_{\tilde{x}}(2+\delta)R)}{\mathfrak{m}(B^+_{\tilde{x}}(R))}\leq \delta^{4n}\frac{v(n,-k^2,(2+\delta)R)}{v(n,-k^2,R)}=:C'(n,k,\delta,R).\tag{3.12}\label{4.4C}
\]
It follows from (\ref{4.3C}) and (\ref{4.4C}) that
\begin{align*}
&\frac{1}{\mathfrak{m}(B^+_{\tilde{x}}(R))}\int_{B^+_{\tilde{x}}(R)}\left((n-1)K-\widetilde{\underline{\mathbf{Ric}}}\right)_+^q d{\mathfrak{m}}\\
\leq& \frac{N}{\mathfrak{m}(B^+_{\tilde{x}}(R))}\int_{M}\left((n-1)K-{\underline{\mathbf{Ric}}}\right)_+^q d{\mathfrak{m}}\\
\leq &C'(n,k,\delta,R)\frac{N}{\mathfrak{m}(B^+_{\tilde{x}}((2+\delta)R))}\int_{M}\left((n-1)K-\underline{\mathbf{Ric}}\right)_+^q d{\mathfrak{m}}\\
\leq &\frac{C'(n,k,\delta,R)}{\mathfrak{m}(M)}\int_{M}\left((n-1)K-\underline{\mathbf{Ric}}\right)_+^q d{\mathfrak{m}}=C'(n,k,\delta,R)\cdot \overline{\mathcal {K}}_{d\mathfrak{m}}(q,K,R),
\end{align*}
which implies that
\[
\widetilde{\overline{{\mathcal {K}}}}_{d\mathfrak{m}}(q,K,R)\leq C'(n,k,\delta,R)\cdot \overline{\mathcal {K}}_{d\mathfrak{m}}(q,K,R).\tag{3.13}\label{4.5C}
\]
Let $\varepsilon=\varepsilon(n,q,k,K,\delta,R,\rho)$ defined in Lemma \ref{Rlesslemm} and let
\[
\overline{\mathcal {K}}_{d\mathfrak{m}}(q,K,R)<\frac{\varepsilon}{ C'(n,k,\delta,R)}.
\]
Now (\ref{4.5C}) together with $\widetilde{\mathbf{Ric}}\geq -(n-1)k^2$, $\Lambda_{\widetilde{F}}\leq \delta^2$ and Lemma \ref{Rlesslemm} yields that $\widetilde{M}$ is compact.

For $R\leq \pi/\sqrt{K}$, one can use the same argument as in Step 1 to show that $\widetilde{M}$ is compact. Therefore, $\pi_1(M)$ is finite.
\end{proof}

For a Berwald manifold, we obtain the following result by means of Petersen-Sprouse \cite{PC}.
\begin{theorem}\label{BreMy}
Given any $n>1$, $q>n/2$,  $K>0$, $R>0$ and $\delta\geq 1$, for each $\rho>0$, there exists $\varepsilon=\varepsilon(n,q, K,\delta,R,\rho)$ such that if
a  (forward) complete Berwald $n$-manifold $(M,F)$ satisfies
\[
\Lambda_F\leq \delta^2,\ \overline{\mathcal {K}}_{d\mathfrak{m}}(q,K,R)<\varepsilon,
\]
then
\[
\diam(M)\leq \delta^2\left( \frac{\pi}{\sqrt{K}}+\rho  \right)
\]
\end{theorem}
\begin{proof}
Let $\hat{g}$ be the average Riemannian metric of $g$ and let $\hat{B}_x(r)$ denote the geodesic ball centered at $x$ with radius $r$ in $(M,\hat{g})$. For each $p\in \hat{B}_x(r)-\{x\}$, there exists a minimal normal geodesic $\gamma(t)$ (with respect to $(M,F)$)
from $x$ to $p$. In a local coordinate system $(x^i)$, set
\[
d\mathfrak{m}=\sigma\cdot dx^1\wedge \cdots\wedge dx^n=h\cdot d\vol,
\]
where $d\vol$ is the Riemannian measure of $\hat{g}$ and
\[
h(\gamma(t))=\frac{\sigma(\gamma(t))}{\sqrt{\det g(\gamma(t),\dot{\gamma}(t))}}\cdot\frac{\sqrt{\det g(\gamma(t),\dot{\gamma}(t))}}{\sqrt{\det \hat{g}(\gamma(t))}}.
\]

In the following, we estimate $h$. Since the Levi-Civita connection of $\hat{g}$ is exactly the Chern connection of $g$, one can choose
a $g_{\dot{\gamma}}$-orthnormal parallel frame field $\{E_i\}$ such that each $E_i|_p$ is the eigenvector of  $\hat{g}|_p$. Note that $h(\gamma(t))$ is independent of the choice of coordinates. Denote by $\det{g}$ and $\det\hat{g}$ be the determinants of $g$ and $\hat{g}$ w.r.t. $\{E_i\}$, respectively. It is easy to see that
\[
\det {g}(\gamma(t),\dot{\gamma}(t))=1,\ \delta^{-2n}\leq \det\hat{g}(\gamma(t))\leq \delta^{2n}.
\]

Then
\[
h(\gamma(t))=e^{-\tau(\dot{\gamma}(t))}\frac{\sqrt{\det g(\gamma(t),\dot{\gamma}(t))}}{\sqrt{\det \hat{g}(\gamma(t))}},
\]
which together with Proposition \ref{uniformtorsion} implies
\[
\delta^{-3n}\leq h(\gamma(t))\leq \delta^{3n}\Rightarrow  \delta^{-3n}d\vol\leq d\mathfrak{m}\leq \delta^{3n}d\vol,\tag{3.14}\label{4.11newnew}
\]
and hence,
\begin{align*}
\frac{1}{\vol(B^+_{x}(R))}\int_{B^+_{x}(R)} f d{\vol}&\leq \delta^{6n}\frac{1}{\mathfrak{m}(B^+_{x}(R))}\int_{B^+_{x}(R)}f d\mathfrak{m}, \tag{3.15}\label{4.7C}
\end{align*}

Note that $\overline{\mathcal {K}}_{d\mathfrak{m}}(q,0,R)\leq \overline{\mathcal {K}}_{d\mathfrak{m}}(q,K,R)$.
Thus, it follows from Theorem \ref{localver} that there exists $\varepsilon_1=\varepsilon_1(n,q,0,\delta,R,\delta^{-4n}/2)$ such that if $\overline{\mathcal {K}}_{d\mathfrak{m}}(q,K,R)<\varepsilon_1$, then for any $x\in M$,
\[
\frac{\mathfrak{m}(B^+_{x}(R))}{\mathfrak{m}(B^+_{x}(R/(1+\delta)))}\leq 2\delta^{4n}(1+\delta)^n.\tag{3.16}\label{4.6C}
\]

Given $p\in M$, let $\{B^+_{x_i}(R/(1+\delta))\}_{i\in I}$ denote the maximal family of disjoint forward balls in $\hat{B}_p(\delta R)$. Thus, $\{B^+_{x_i}(R)\}_{i\in I}$ is a covering of $\hat{B}_p(\delta R)$. Let $\mathcal {J}(p)$ denote the minimal eigenvalue of the Ricci tensor of $\hat{g}$ at $p\in M$. Since $(M,F)$ is Berwaldian, one gets
$\delta^{-2}\left((n-1)K\delta^{-2}-\mathcal {J}\right)_+\leq \left((n-1)K-\underline{\mathbf{Ric}}\right)_+$.
Thus,
it follows from (\ref{4.11newnew})-(\ref{4.6C}) that
\begin{align*}
&\frac{\delta^{-2q}}{\vol(\hat{B}_p(\delta R))}{\int_{\hat{B}_p(\delta R)}}\left((n-1)\frac{K}{\delta^{2}}-\mathcal {J}\right)_+^q d\vol\\
\leq &\frac{\delta^{6n}}{\mathfrak{m}(\hat{B}_p(\delta R))}\sum_{i\in I}\int_{B^+_{x_i}(R)}\left((n-1){K}-\underline{\mathbf{Ric}}\right)_+^q d\mathfrak{m}\\
\leq &\delta^{6n}\sum_{i\in I}\frac{\mathfrak{m}(B^+_{x_i}(R))}{\mathfrak{m}(B^+_{x_i}(R/(1+\delta)))}\frac{\mathfrak{m}(B^+_{x_i}(R/(1+\delta)))}{\mathfrak{m}(\hat{B}_p(\delta R))}\overline{\mathcal {K}}_{d\mathfrak{m}}(q,K,R)\\
\leq &\delta^{6n}\sup_{i\in I}\frac{\mathfrak{m}(B^+_{x_i}(R))}{\mathfrak{m}(B^+_{x_i}(R/(1+\delta)))}\cdot \overline{\mathcal {K}}_{d\mathfrak{m}}(q,K,R)\leq 2\delta^{10n}(1+\delta)^n\cdot \overline{\mathcal {K}}_{d\mathfrak{m}}(q,K,R).
\end{align*}
That is,
\[
\widehat{\overline{\mathcal {K}}}_{d\vol}(q,K/\delta^2,\delta R)\leq  2\delta^{10n+2q}(1+\delta)^n\cdot \overline{\mathcal {K}}_{d\mathfrak{m}}(q,K,R),\tag{3.17}\label{4.14newnew}
\]
where $\widehat{\overline{\mathcal {K}}}_{d\vol}(\cdot,\cdot,\cdot)$ is the integral curvature of $(M,\hat{g},d\vol)$.

According to \cite[Theorem 1.1]{PC}, there exists $\varepsilon_2=\varepsilon_2(n,p,K,\delta,R,\rho)$ such that if $\widehat{\overline{\mathcal {K}}}_{d\vol}(q,K/\delta^2,\delta R)<\varepsilon_2$,
then
\[
\diam_{\hat{g}}(M)\leq \delta\left( \frac{\pi}{\sqrt{K}}+\rho  \right).
\]
Due to (\ref{4.14newnew}), we conclude the proof by choosing
\[
\varepsilon:=\min\left\{\varepsilon_1,\frac{\varepsilon_2}{ 2\delta^{10n+2q}(1+\delta)^n} \right\}.
\]
\end{proof}

\noindent{\it The proof of Theorem \ref{mysfirst}.}
Theorem \ref{mysfirst} follows from Theorem \ref{firstheorem} and Theorem \ref{BreMy} directly.

\section{Appendix}

In this section, we will establish a relative volume comparison by the integral curvature.
Now we recall the polar coordinate system of a Finsler manifold first.

Fix a point $p\in M$ arbitrarily.
Set
$D_p:=\{ty\in T_pM: y\in S_pM, 0\leq t<i_y\}$, and $\mathcal
{D}_p:=\exp_p(D_p)$. Thus, $M=\mathcal {D}_p\sqcup \text{Cut}_p$
(cf. \cite[Prop. 8.5.2]{BCS}). Now, let $r(x):=d(p,x)$. Then
 the {\it polar coordinates} at $p$ actually describe a
diffeomorphism of $D_p-\{0\}$ onto $\mathcal
{D}_p\backslash\{p\}$, given by
\[
(r,y)\mapsto \exp_p ry.
\]
Now write
\[
d\mathfrak{m}=:\hat{\sigma}_p(r,y)dr\wedge d\nu_p(y),
\]
where $d\nu_p(y)$ is the Riemannian volume measure induced by $F$ on $S_pM$. According to \cite{Sh2,W,ZY2}, one has
\[
\left.\nabla r\right|_{(r,y)}=\left.\frac{\partial}{\partial r}\right|_{(r,y)}=\dot{\gamma}_y(r),\ \Delta r=\frac{\partial}{\partial r}\log \hat{\sigma}_p(r,y),\ \text{ for }0<r<i_y,\tag{A.1}\label{1.2}
\]
and
\[
\lim_{r\rightarrow 0^+}\frac{\hat{\sigma}_p(r,y)}{e^{-\tau(\dot{\gamma}_y(r))}\mathfrak{s}^{n-1}_k(r)}=1,\tag{A.2}\label{new2.3}
\]
where $\gamma_y(t):=\exp_p(ty)$  and $\mathfrak{s}_k(t)$ is the unique
solution to $f'' + kf = 0$ with $f(0) = 0$ and $f'(0) = 1$.

It is easy to check that (also see \cite[(4.5)]{Wu2})
\[
\frac{\partial}{\partial r}H\leq -\mathbf{Ric}(\nabla r)-\frac{H^2}{n-1},\ \text{ for any }(r,y)\in D_p,
\]
where
\[
H(r,y):=\frac{\partial}{\partial r}\log \left[ \hat{\sigma}_p(r,y) e^{\tau(\dot{\gamma}_y(r))}  \right]=\frac{\partial}{\partial r}\log \sqrt{\det g_{\nabla r}}.
\]
Also set
\[
H_k(r):=\frac{\partial}{\partial r}\log \mathfrak{s}^{n-1}_k(r).
\]
Thus, (\ref{new2.3}) reads
\[
\lim_{r\rightarrow 0^+}\left[ H(r,y)-H_k(r)\right]=0.\tag{A.3}\label{nA1}
\]

Inspired by \cite{PW}, we have following theorem.
\begin{theorem}\label{localver}
Given $\delta\geq 1$, $q>n/2$ and $k\leq 0$, for  $\alpha\in(0,\delta^{-4n})$,  there exists an $\varepsilon=\varepsilon(n,q,k,\delta,R,\alpha)>0$ such that
any forward  complete  Finsler $n$-manifold $(M,F,d\mathfrak{m})$  with
\[
\Lambda_F\leq \delta^2,\ \overline{\mathcal {K}}_{d\mathfrak{m}}(q,k,R)<\varepsilon,
\]
must satisfy
\[
\alpha \cdot\frac{v(n,k,r_1)}{v(n,k,r_2)}\leq \frac{\mathfrak{m}(B^+_{x}(r_1))}{\mathfrak{m}(B^+_{x}(r_2))},\tag{A.4}\label{3.4newnew}
\]
for
 all $x\in M$ and $0<r_1\leq r_2\leq R$.

Moreover,
 if additionally suppose that $F$ is Berwaldian, then the above result still holds without the assumption $\Lambda_F\leq \delta^2$, in which case $\alpha\in(0,1)$.
\end{theorem}

\begin{proof}
\noindent\textbf{Step 1.}
Define two functions on $[0,+\infty)\times S_pM$ as follows:
\begin{align*}
\Psi(r,y)&:=\left\{
\begin{array}{lll}
& (H(r,y)-H_k(r) )_+, & \ \ \ \text{if } 0\leq r<i_y, \ y\in S_pM \\
\\
&0, & \ \ \ \text{if } r\geq i_y,
\end{array}
\right.\\
\\
e^{\tau(r,y)}&:=\left\{
\begin{array}{lll}
& e^{\tau(\dot{\gamma}_y(r))}, & \ \ \ \text{if } 0\leq r<i_y, \ y\in S_pM , \\
\\
&0, & \ \ \ \text{if } r\geq i_y.
\end{array}
\right.
\end{align*}
Thus, given $y\in S_pM$, for almost every $r>0$, we have
\[
\frac{d}{dr}\left[\frac{e^{\tau(r,y)}\hat{\sigma}_p(r,y)}{\mathfrak{s}^{n-1}_k(r)} \right]\leq \Psi(r,y)\frac{e^{\tau(r,y)}\hat{\sigma}_p(r,y)}{\mathfrak{s}^{n-1}_k(r)},\tag{A.5}\label{nA0}
\]
which furnishes that for all $0\leq t<r<+\infty$,
\begin{align*}
\mathfrak{s}_k^{n-1}(t)\cdot e^{\tau(r,y)}\hat{\sigma}_p(r,y)-\mathfrak{s}_k^{n-1}(r)\cdot e^{\tau(t,y)}\hat{\sigma}_p(t,y)
\leq \mathfrak{s}_k^{n-1}(r)\int^r_0\Psi(s,y){e^{\tau(s,y)}\hat{\sigma}_p(s,y)}ds.
\end{align*}
Now the above inequality together with the H\"older inequality yields
\begin{align*}
&\frac{d}{dr}\left[\frac{\int_{B^+_p(r)}e^\tau d\mathfrak{m}}{v(n,k,r)}\right]\\
=&\frac{\vol(\mathbb{S}^{n-1})\int_{S_pM}\int^r_0\left[\mathfrak{s}_k^{n-1}(t)e^{\tau(r,y)}\hat{\sigma}_p(r,y)-\mathfrak{s}_k^{n-1}(r) e^{\tau(t,y)}\hat{\sigma}_p(t,y) \right]dt\, d\nu_p(y)}{v(n,k,r)^2}\\
\leq &\frac{\vol(\mathbb{S}^{n-1})\cdot r\cdot\mathfrak{s}_k^{n-1}(r)\int_{S_pM}\int^r_0\Psi(s,y){e^{\tau(s,y)}\hat{\sigma}_p(s,y)}ds\, d\nu_p(y)}{v(n,k,r)^2}\\
\leq &\frac{{\vol(\mathbb{S}^{n-1})}\cdot r\cdot\mathfrak{s}_k^{n-1}(r)\cdot\left(\int_{B^+_p(r)} \Psi^{2q} e^{\tau} d\mathfrak{m}\right)^{1/2q}\left(\int_{B^+_p(r)}e^\tau d\mathfrak{m}\right)^{1-1/2q}}{v(n,k,r)^2}.
\end{align*}
That is,
\begin{align*}
&\frac{d}{dr}\left( \frac{\int_{B^+_p(r)}e^\tau d\mathfrak{m}}{v(n,k,r)}\right)\\
\leq& C_1(n,k,r)\cdot \left( \frac{\int_{B^+_p(r)}e^\tau d\mathfrak{m}}{v(n,k,r)} \right)^{1-\frac{1}{2q}}\cdot\left(\int_{B^+_p(r)} \Psi^{2q} e^{\tau} d\mathfrak{m} \right)^{\frac{1}{2q}}\cdot v(n,k,r)^{-\frac{1}{2q}},\tag{A.6}\label{3.5newnwew}
\end{align*}
where
\[
C_1(n,k,r):=\max_{t\in [0,r]}\frac{\vol(\mathbb{S}^{n-1})\cdot t\cdot \mathfrak{s}^{n-1}_k(t) }{v(n,k,t)}<\infty.
\]

\noindent \textbf{Step 2.} Now we claim that there exists $C_2(n,q)>0$ such that if $q>n/2$, then
\[
\int^r_0 \Psi^{2q} e^{\tau(t,y)}\hat{\sigma}_p(t,y)dt\leq C_2(n,q)\int^r_0 \rho^q e^{\tau(t,y)}\hat{\sigma}_p(t,y) dt,\ \forall \,r>0,\tag{A.7}\label{3.6newnew}
\]
where
\begin{align*}
\rho(r,y)&:=\left\{
\begin{array}{lll}
& ( (n-1)k-\mathbf{Ric}(\nabla r))_+, & \ \ \ \text{if } 0<r<i_y, \ y\in S_pM, \\
\\
&0, & \ \ \ \text{if }r\geq i_y \text{ or } r=0.
\end{array}
\right.
\end{align*}

In fact, the definition of $\Psi$ yields that
\[
\frac{\partial}{\partial r}\Psi+\frac{\Psi^2}{n-1}+2\frac{\Psi\cdot H_k}{n-1}\leq \rho, \text{ for almost every }r>0,
\]
which implies
\begin{align*}
&\int^r_0\frac{\partial}{\partial t}\Psi\cdot\Psi^{2q-2}e^{\tau}\hat{\sigma}_p dt+\frac{1}{n-1}\int^r_0 \Psi^{2q}e^{\tau}\hat{\sigma}_pdt+\frac{2}{n-1}\int^r_0H_k \Psi^{2q-1}e^{\tau}\hat{\sigma}_pdt\\
&\leq \int^r_0 \rho\cdot\Psi^{2q-2}e^{\tau}\hat{\sigma}_pdt\tag{A.8}\label{nA2}.
\end{align*}
Note that (\ref{nA1}) implies $\Psi(0)=0$ and $\Psi\geq 0$. Thus,
the first item of (\ref{nA2}) together with the definitions of $H$ and $\Psi$ furnishes
\begin{align*}
\int^r_0 \frac{\partial}{\partial t}\Psi\cdot\Psi^{2q-2}e^{\tau}\hat{\sigma}_p dt
\geq-\frac{1}{2q-1}\int^r_0 \Psi^{2q-1} (H_k+\Psi) e^\tau \hat{\sigma}_pdt,
\end{align*}
which together with (\ref{nA2}) yields
\begin{align*}
&\left(\frac{1}{n-1}-\frac{1}{2q-1} \right)\int^r_0  \Psi^{2q}e^{\tau}\hat{\sigma}_p dt+\left(\frac{2}{n-1}-\frac{1}{2q-1} \right)\int^r_0  \Psi^{2q-1}\,H_k\,e^{\tau}\hat{\sigma}_p dt\\
&\leq \int^r_0 \rho\, \Psi^{2q-2}\,e^{\tau}\hat{\sigma}_p dt.
\end{align*}
Since $q>n/2$,  the above inequality yields
\begin{align*}
&\left(\frac{1}{n-1}-\frac{1}{2q-1} \right)\int^r_0  \Psi^{2q}e^{\tau}\hat{\sigma}_p dt\leq \int^r_0 \rho \Psi^{2q-2}\,e^{\tau}\hat{\sigma}_p dt\\
\leq &\left(\int^r_0 \rho^q \,e^{\tau}\hat{\sigma}_p dt \right)^{1/q}\cdot \left(\int^r_0 \Psi^{2q}\,e^{\tau}\hat{\sigma}_p dt \right)^{1-1/q},
\end{align*}
which implies the claim (\ref{3.6newnew}) is true. Now  (\ref{3.6newnew})
 together with (\ref{3.5newnwew}) furnishes
\begin{align*}
&\frac{d}{dr}\left( \frac{\int_{B^+_p(r)}e^\tau d\mathfrak{m}}{v(n,k,r)}\right)\\
\leq& C_3(n,q,k,r)\cdot \left( \frac{\int_{B^+_p(r)}e^\tau d\mathfrak{m}}{v(n,k,r)} \right)^{1-\frac{1}{2q}}\cdot\left(k_p(q,k,r)\right)^{\frac{1}{2q}}\cdot v(n,k,r)^{-\frac{1}{2q}},\tag{A.9}\label{3.8newnew}
\end{align*}
where $C_3(n,q,k,r):=C_1(n,k,r)\,C_2(n,q)^{\frac1{2q}}$   and
\begin{align*}
k_p(q,k,r)&:=\int_{B^+_p(r)}\rho^q e^{\tau}d\mathfrak{m}=\int_{S_pM} d\nu_p(y)\int^r_0 \rho^q(t,y) e^{\tau(t,y)}\hat{\sigma}_p(t,y)dt .
\end{align*}

\noindent \textbf{Step 3.} Now set
\[
h(r):=\frac{\int_{B^+_p(r)}e^\tau d\mathfrak{m}}{v(n,k,r)},\ f(r):= C_3(n,q,k,r)\cdot\left(k_p(q,k,r)\right)^{\frac{1}{2q}}\cdot v(n,k,r)^{-\frac{1}{2q}}.
\]
Thus, (\ref{3.8newnew}) together with  (\ref{new2.3}) furnishes $h'\leq h^{1-\frac{1}{2q}}\cdot f(r)$.
Hence, for any $0<r_1\leq r_2\leq R$,
\[
2q\cdot h^{\frac{1}{2q}}(r_2)-2q \cdot h^{\frac{1}{2p}}(r_1)\leq \int^{r_2}_{r_1} f(s)ds.
\]
Since $C_3(n,q,k,r)$ and $k_p(q,k,r)$ are nondecreasing in $r$, one has
\begin{align*}
\int^{r_2}_{r_1} f(s)ds
\leq C_3(n,q,k,R)\cdot\int^R_0 v(n,k,s)^{-\frac{1}{2q}} ds\cdot\left(k_p(q,k,R)\right)^{\frac{1}{2q}}.
\end{align*}
Set
\[
C_4(n,q,k,R):=\frac{1}{2q}C_3(n,q,k,R)\cdot\int^R_0 v(n,k,s)^{-\frac{1}{2q}} ds.
\]
$C_4(n,q,k,R)$ is well-defined if $q>n/2$. Then we obtain
\[
h^{\frac1{2q}}(r_2)-h^{\frac1{2q}}(r_1)\leq C_4(n,q,k,R)\cdot\left(k_p(q,k,R)\right)^{\frac{1}{2q}},
\]
which together with Proposition \ref{uniformtorsion} yields
\[
h^{\frac1{2q}}(r_2)-h^{\frac1{2q}}(r_1)\leq C_4(n,q,k,R)\cdot\left(\delta^{2n}\cdot \mathcal {K}_{p,d\mathfrak{m}}(k,q,R)\right)^{\frac{1}{2q}},\tag{A.10}\label{**3}
\]
where
\[
\mathcal {K}_{p,d\mathfrak{m}}(k,q,R):=\int_{B^+_p(R)} \left((n-1)k-\underline{\mathbf{Ric}}(x)\right)_+^q d\mathfrak{m}(x)
\]

\noindent \textbf{Step 4.} It is not hard to see that (\ref{**3}) implies
\begin{align*}
&\left( \frac{v(n,k,r_1)}{v(n,k,r_2)}\right)^{\frac1{2q}}-\left(  \frac{\int_{B^+_p(r_1)}e^\tau d\mathfrak{m}}{\int_{B^+_p(r_2)}e^\tau d\mathfrak{m}} \right)^{\frac{1}{2q}}\\
\leq& C_4(n,q,k,R)\cdot \delta^{n/q}\cdot \mathcal {K}_{p,d\mathfrak{m}}(k,q,R)^{\frac1{2q}} \cdot \left( \frac{v(n,k,r_1)}{v(n,k,r_2)}\right)^{\frac1{2q}}\cdot\left(\frac{v(n,k,r_2)}{\int_{B^+_p(r_2)}e^\tau d\mathfrak{m}} \right)^{\frac{1}{2q}},
\end{align*}
which together with Proposition \ref{uniformtorsion} furnishes
\[
\delta^{-4n}(1-c)^{2q}\frac{v(n,k,r_1)}{v(n,k,r_2)}\leq \frac{\mathfrak{m}(B^+_p(r_1))}{\mathfrak{m}(B^+_p(r_2))},\tag{A.11}\label{3.100newnew}
\]
where
\[
c:= C_4(n,q,k,R)\cdot \delta^{n/q}\cdot \mathcal {K}_{p,d\mathfrak{m}}(k,q,R)^{\frac1{2q}} \cdot \left(\frac{v(n,k,r_2)}{\int_{B^+_p(r_2)}e^\tau d\mathfrak{m}} \right)^{\frac{1}{2q}}.\tag{A.12}\label{3.10newnew}
\]
In order to estimate $c$, we use (\ref{**3}) again ($r_2\rightarrow R$, $r_1\rightarrow r_2$) and obtain
\begin{align*}
&\left( \frac{v(n,k,r_2)}{\int_{B^+_p(r_2)}e^\tau d\mathfrak{m} } \right)^{\frac1{2q}}\leq \left(\left( \frac{\int_{B^+_p(R)}e^\tau d\mathfrak{m}}{v(n,k,R)} \right)^{\frac1{2q}}-C_4(n,q,k,R)\cdot \delta^{n/q}\cdot \mathcal {K}_{p,d\mathfrak{m}}(k,q,R)^{\frac1{2q}}   \right)^{-1}\\
\leq & \left( \frac{v(n,k,R)}{\int_{B^+_p(R)}e^\tau d\mathfrak{m}} \right)^{\frac1{2q}} \left( 1-C_4(n,q,k,R)\cdot v(n,k,R)^{\frac{1}{2q}}\cdot  \delta^{2n/q}\cdot \overline{\mathcal {K}}^{\frac1{2q}}_{d\mathfrak{m}}(k,q,R)  \right)^{-1}.
\end{align*}
Hence, there exists an $\varepsilon_1=\varepsilon_1(n,q,k,\delta,R)>0$ such that if $ \overline{\mathcal {K}}_{d\mathfrak{m}}(k,q,R)<\varepsilon_1$, then
\[
\left( \frac{v(n,k,r_2)}{\int_{B^+_p(r_2)}e^\tau d\mathfrak{m} } \right)^{\frac1{2q}}\leq   2\left( \frac{v(n,k,R))}{\int_{B^+_p(R)}e^\tau d\mathfrak{m}} \right)^{\frac1{2q}}.\tag{A.13}\label{3.11newnew}
\]
On the other hand,  by (\ref{3.10newnew}), (\ref{3.11newnew}) and Proposition \ref{uniformtorsion}, we can choose $\varepsilon_2=\varepsilon_2(n,q,k,\delta,R,\alpha)>0$ such that if $\overline{\mathcal {K}}_{d\mathfrak{m}}(k,q,R)<\varepsilon_2$,
\[
c\leq  C_5(n,q,k,R)\cdot \delta^{2n/q}\cdot  \varepsilon_2^{\frac1{2q}}\leq 1-\alpha^{\frac1{2q}}\delta^{\frac{2n}q},\tag{A.14}\label{3.14newnew}
\]
where $C_5(n,q,k,R):=2\,C_4(n,q,k,R)\cdot v(n,k,R)^{1/(2q)}$.   Choosing $\varepsilon:=\min\{\varepsilon_1,\varepsilon_2\}$, we obtain (\ref{3.4newnew}) by (\ref{3.100newnew}) and (\ref{3.14newnew}) directly.

\noindent \textbf{Step 5.} Now additionally suppose that  $F$ is Berwaldian. It follows from \cite{Sh1} that the S-curvature of $d\mathfrak{m}$ always vanishes, which implies that the distortion $\tau(\dot{\gamma}_y(r))$ only depends on $y$. In particular, we have
\[
H=\frac{\partial}{\partial r}\log \left[ \hat{\sigma}_p(r,y)e^{\tau(\dot{\gamma}_y(r))} \right]=\frac{\partial}{\partial r}\log \hat{\sigma}_p(r,y).
\]
Now we extend $\hat{\sigma}_p$   by setting $\hat{\sigma}_p(r,y):=0$ if $r\geq i_y$.
Thus, (\ref{nA0})
yields that, given $y\in S_pM$, on almost every $r>0$
\[
\frac{d}{dr}\left[\frac{\hat{\sigma}_p(r,y)}{\mathfrak{s}^{n-1}_k(r)} \right]\leq \Psi(r,y)\frac{\hat{\sigma}_p(r,y)}{\mathfrak{s}^{n-1}_k(r)}.
\]
The same argument yields the result.
\end{proof}

\begin{remark}
Theorem \ref{localver} can be extended to any measure  if the assumption $\Lambda_F\leq \delta^2$ is replaced by $a\leq \tau\leq  b$.
\end{remark}

We note that one can obtain some precompactness theorems and finiteness theorems in the Finsler setting by Theorem \ref{localver}.
We leave these statements to the interested reader.

\section{Acknowledgements}
This work was done while the author was a visiting scholar at
IUPUI.  The author thanks Professor Zhongmin Shen for his hospitality. This work was supported by NNSFC (No. 11501202, No. 11761058) and the grant of China Scholarship Council.

\bibliography{mybibfile}

\end{document}